\documentclass{amsart}

\usepackage{amssymb,amstext,amsmath,amscd,amsthm,amsfonts,enumerate,latexsym}
\usepackage{color}
\usepackage[dvipdfmx]{graphicx}
\usepackage[all]{xy}
\usepackage{extarrows}

\usepackage{ytableau}

\usepackage{array}
\usepackage{xcolor}
\usepackage[leqno]{amsmath}
\theoremstyle{plain}
\newtheorem{thm}{Theorem}[section]
\newtheorem*{thm*}{Theorem}
\newtheorem*{cor*}{Corollary}

\newtheorem{prop}[thm]{Proposition}

\newtheorem{cor}[thm]{Corollary}

\newtheorem*{claim*}{Claim}

\theoremstyle{definition}
\newtheorem{defn}[thm]{Definition}
\newtheorem{ex}[thm]{Example}

\newtheorem{notation}[thm]{Notation}
\newtheorem{setup}[thm]{Setup}

\theoremstyle{remark}

\numberwithin{equation}{thm}
\def\Hom{\mathrm{Hom}}
\def\Max{\mathrm{Max}}

\def\e{\mathrm{e}}
\def\m{\mathfrak m}
\def\M{\mathcal {M}}
\def\N{\mathcal {N}}
\def\n{\mathfrak n}

\def\q{\mathfrak q}

\def\K{\mathrm{K}}

\newcommand{\ed}{\operatorname{embdim}}

\newcommand{\rme}{\mathrm{e}}

\newcommand{\rmK}{\mathrm{K}}

\newcommand{\rmQ}{\mathrm{Q}}

\newcommand{\calF}{\mathcal{F}}

\newcommand{\fka}{\mathfrak{a}}

\newcommand{\fkc}{\mathfrak{c}}

\newcommand{\numberseries}{\bfseries}   %Fontseries used for numbering

\newlength{\thmtopspace}                %Space above theorem
\newlength{\thmbotspace}                %Space below theorem
\newlength{\thmheadspace}               %Space after theorem label
\newlength{\thmindent}                  %For indenting

\newtheoremstyle{fixed bf head,slanted body}
                {\thmtopspace}{\thmbotspace}{\slshape}
                {\thmindent}{\bfseries}{.}{\thmheadspace}
                {{\numberseries \thmnumber{#2\;}}\thmname{#1}\thmnote{ (#3)}}

\newtheoremstyle{variable bf head,slanted body}
                {\thmtopspace}{\thmbotspace}{\slshape}
                {\thmindent}{\bfseries}{.}{\thmheadspace}
                {{\numberseries \thmnumber{#2\;}}\thmname{#1}\thmnote{ #3}}

\newtheoremstyle{fixed bf head,upright body}
                {\thmtopspace}{\thmbotspace}{\upshape}
                {\thmindent}{\bfseries}{.}{\thmheadspace}
                {{\numberseries \thmnumber{#2\;}}\thmname{#1}\thmnote{ (#3)}}

\newtheoremstyle{numbered paragraph}
                {\thmtopspace}{\thmbotspace}{\upshape}
                {\thmindent}{\upshape}{}{\thmheadspace}
                {{\numberseries \thmnumber{#2.}}}

\newcommand{\Frac}{\operatorname{Frac}}

\begin{document}

\setlength{\baselineskip}{13pt}
%%%%%%%%%%%%%%%%%%%%%%%%%%%%%%%%%%%%%%%%%%%%%%%%%%%%%%%%%%%%%

\title{Generalized Gorenstein Arf rings}

\author[Celikbas, Celikbs, Goto, Taniguchi]
{Ela Celikbas, Olgur Celikbas, \\ Shiro Goto, and Naoki Taniguchi}

\address{Ela Celikbas \\
Department of Mathematics \\
West Virginia University\\
Morgantown, WV 26506-6310 USA}
\email{ela.celikbas@math.wvu.edu}

\address{Olgur Celikbas \\
Department of Mathematics \\
West Virginia University\\
Morgantown, WV 26506-6310 USA}
\email{olgur.celikbas@math.wvu.edu}

\address{Shiro Goto \\
Department of Mathematics\\
School of Science and Technology\\
Meiji University, 1-1-1 Higashi-mita, Tama-ku, Kawasaki 214-8571, Japan}
\email{shirogoto@gmail.com}

\address{Naoki Taniguchi\\
Global Education Center\\
Waseda University\\
1-6-1 Nishi-Waseda, Shinjuku-ku, Tokyo 169-8050, Japan}
\email{naoki.taniguchi@aoni.waseda.jp}
\urladdr{http://www.aoni.waseda.jp/naoki.taniguchi/}



\maketitle

%%%%%%%%%%%%%%%%%%%%%%%%%%%%%%%%%%%%%%%%%%%%%%%%%%%%%%%%%%%%
%%%%%%%%%%%%%%%%%%%%%%%%%%%%%%%%%%%%%%%%%%%%%%%%%%%%%%%%%%%%

\begin{abstract} 
In this paper we study generalized Gorenstein Arf rings; a class of one-dimensional Cohen-Macaulay local Arf rings that is strictly contained in the class of Gorenstein rings. We obtain  new characterizations and examples of Arf rings, and give applications of our argument to numerical semigroup rings and certain idealizations. In particular, we generalize a beautiful result of Barucci and Fr\"oberg concerning Arf numerical semigroup rings.
\end{abstract}

%%%%%%%%%%%%%%%%%%%%%%%%%%%%%%%%%%%%%%%%%%%%%%%%%%%%%%%%%%%%
%%%%%%%%%%%%%%%%%%%%%%%%%%%%%%%%%%%%%%%%%%%%%%%%%%%%%%%%%%%%%%%%%%%%%%%%%%%%%%%%%%%%%%%%%%%%%%%%%%%%%%%%%%%%%%%%%%%%%%%%
%%%%%%%%%%%%%%%%%%%%%%%%%%%%%%%%%%%%%%%%%%%%%%%%%%%%%%%%%%%%%%%%%%%%%%%%%%%%%%%%%%%%%%%%%%%%%%%%%%%%%%%%%%%%%%%%%%%%%%%%
%%%%%%%%%%%%%%%%%%%%%%%%%%%%%%%%%%%%%%%%%%%%%%%%%%%%%%%%%%%%%%%%%%%%%%%%%%%%%%%%%%%%%%%%%%%%%%%%%%%%%%%%%%%%%%%%%%%%%%%%
%%%%%%%%%%%%%%%%%%%%%%%%%%%%%%%%%%%%%%%%%%%%%%%%%%%%%%%%%%%%

\section{Introduction}
In 1971 Lipman \cite{Lipman} proved that, if $(R, \m)$ is a complete, one-dimensional local domain with an algebraically closed field of characteristic zero, and $R$ is saturated (as defined by Zariski in \cite{Zariski}), then $R$ has minimal multiplicity, i.e., the embedding dimension of $R$ equals the multiplicity of $R$.  Lipman's proof based on the fact that such a ring $R$ is an \emph{Arf} ring, i.e., $R$ satisfies a certain condition that was studied by Arf \cite{Arf} in 1949 pertaining to a certain classification of curve singularities; see, for example, the survey papers \cite{Sertoz1}, \cite{Sertoz2} and also Du Val's work \cite{Duval} for details. As Lipman \cite{Lipman} pointed out, the defining condition of an Arf ring is technical, but it is convenient to work with, as well as, easy to state: if $R$ is as above, then $R$ is Arf provided that, $yz/x \in R$ whenever $0\neq x \in \m$ and $y/x, z/x \in \Frac(R)$ (the field of fractions of $R$) are integral elements over $R$.

Examples of Arf rings are abundant. For example, $R$ is Arf if the multiplicty of $R$ is at most two. 
As Arf property is preserved by standard procedures in ring theory, such as completion, it is not difficult to construct examples of Arf rings; see, for example, \cite[2.5 and 2.7]{Lipman}. 

Arf property for numerical semigroup rings, as well as algorithms to compute the Arf ring closure of various rings, such as the coordinate rings of curves, were already studied in the literature; see, for example, \cite{AS, BF, RGGB}. However, to our best knowledge, a homological characterization of Arf rings in local algebra -- besides Lipman's beautiful work -- is yet to be given. The main purpose of this paper, rather than seeking fast algorithms to compute the Arf closure, is to initiate a homological investigation, and to attempt to motivate the interested reader to study further in this direction. Although it has seemingly a non-homological definition, Arf rings enjoy important homological properties: if $R$ is Arf, then it has minimal multiplicity so that it is \emph{Golod}, a class of (local) rings which is of active research interest; see, for example, \cite[5.2.8]{Lucho}. This particular property of Arf rings naturally raises the following question: if $R$ has minimal multiplicity, then, under what conditions, $R$ is an Arf ring?
In this paper we are able to give an answer to this question and obtain new characterizations of a class of Arf rings. Our main result is:

\begin{thm}\label{intro1.2} Let $R$ be a one-dimensional generalized Gorenstein local ring with a canonical ideal $I$ which contains the parameter ideal $(a)$ as a reduction. Set $S=\left\{\frac{x}{a} ~\middle|~ x \in I\right\}\subseteq \rmQ(R)$, where $\rmQ(R)$ is the total quotient ring of $R$. Then $R$ is Arf if and only if $R$ has minimal multiplicity and the multiplicity of  $S_{\mathcal{M}}$ is at most two for each maximal ideal $\mathcal{M}$ of $S$.
%\begin{enumerate}[\rm(i)]
%\item $R$ is an Arf ring.
%\item $R$ has minimal multiplicity and $\e(S_{\mathcal{M}})\leq 2$ for each maximal ideal $\mathcal{M}$ of $S$.
%\end{enumerate}
\end{thm}

A generalized Gorenstein ring \cite{GK} is one of the generalization of a Gorenstein ring, defined by a certain embedding of the rings into their canonical modules; see \ref{def1.2} for the precise definition. The class of generalized Gorenstein rings is a new class of Cohen-Macaulay rings, which naturally covers the class of Gorenstein rings and fills the gap in-between Cohen-Macaulay and Gorenstein properties; see \cite{CGKM, GGHV, GK, GK2, GMP, GMTY1, GMTY2, GMTY3, GMTY4, GRTT, GTT, GTT2, GT, T}.  In fact such rings extend the definition of almost Gorenstein rings which were initially defined by Barucci and Fr\"oberg \cite{BF} over one-dimensional analytically unramified local rings, and further developed  and defined by Goto, Matsuoka, and Phuong \cite{GMP} over arbitrary Cohen-Macaulay local rings of dimension one. 

The next two corollaries of Theorem \ref{intro1.2} yield generalizations of a characterization of Barucci and Fr\"oberg \cite[13]{BF} concerning Arf numerical semigroup rings; see Theorem \ref{3.5}, Corollary \ref{3.4}, and Proposition \ref{3.7}. We set $B = \m:_{\rmQ(R)}\m$, the endomorphism algebra of $\m$, where $\rmQ(R)$ is the total quotient ring of $R$. Note that, if $R$ is not Gorenstein but almost Gorensein, then $B=S$; see \cite[3.16]{GMP}.

\begin{cor}\label{intro 3.4} Let $(R, \m)$ be a one-dimensional Cohen--Macaulay local ring with canonical module. Then the following conditions are equivalent.
\begin{enumerate}[\rm(i)]
\item $R$ is an almost Gorenstein Arf ring.
\item $\e(B_{\mathcal{M}}) \le 2$ for each maximal ideal $\mathcal{M}$ of $B$.
\end{enumerate}
\end{cor}

Given a generalized Gorenstein numerical semigroup ring, our next result may be used to check whether it is Arf; see Proposition \ref{3.7}. We set $\fkc = R:_{\rmQ(R)}S$.

\begin{cor} Let $\ell>0$ be an integer and $0 < a_1 < a_2 < \cdots < a_{\ell}$ be integers with $\gcd(a_1, a_2, \ldots, a_{\ell})=1$. Let $k$ be a field, $R=k[\![t^{a_1}, t^{a_2}, \ldots, t^{a_\ell}]\!]$ be the numerical semigroup ring over $k$, and let $H=\left<a_1, a_2, \ldots, a_{\ell}\right>$ be the corresponding semigroup. Assume $R$ is a generalized Gorenstein ring. Then the following  are equivalent.
\begin{enumerate}[\rm(i)]
\item $R$ is an Arf ring.
\item $R$ has minimal multiplicity, $2+ \ell_R(R/\fkc)\cdot a_1 \in H$, and $2 + a_i \in H$ for each $i=2, \ldots, \ell$. 
\end{enumerate}
\end{cor}

In section 4 we consider idealizations of the form $A = R \ltimes \fkc$, where $R$ is an one-dimensional Cohen-Macaulay local ring. We obtain a new criterion for $A$ to be a generalized Gorenstein Arf ring in terms of the integral closure $\overline{R}$ of $R$. A special case of our result is as follows; see Theorem \ref{4.1}.

\begin{cor} Let $R$ be a one-dimensional Cohen-Macaulay local ring $R$ with a canonical module. If $R$ is a generalized Gorenstein ring that has minimal multiplicity, and $S=\overline{R}$, then both $R$ and $A$ are generalized Gorenstein Arf rings.
\end{cor}

%%%%%%%%%%%%%%%%%%%%%%%%%%%%%%%%%%%%%%%%%%%%%%%%%%%%%%%%%%%%%%%%%%%%%%%%%%%%%%%%%%%%%%%%%%%%%%%%%%%%%%%%%%%%%%%%%%%%%%%%%%%%%%%%%%%%%%%%%%%%%%%%%%%%%%%%%%%%%%%%%%%%%%%%%%%%%%%%%%%%%%%%%%%%%%%%%%%%%%%%%%%%%%
%%%%%%%%%%%%%%%%%%%%%%%%%%%%%%%%%%%%%%%%%%%%%%%%%%%%%%%%%%%%%%%%%%%%%%%%%%%%%%%%%%%%%%%%%%%%%%%%%%%%%%%%%%%%%%%%%%%%%%%%%%%%%%%%%%%%%%%%%%%%%%%%%%%%%%%%%%%%
%%%%%%%%%%%%%%%%%%%%%%%%%%%%%%%%%%%%%%%%%%%%%%%%%%%%%%%%%%%%%%%%%%%%%%%%%%%%%%%%%%%%%%%%%%%%%%%%%%%%%%%%%%%%%%%%%%%%%%%%%%%%%%%%%%%%%%%%%%%%%%%%%%%%%%%%%%%%
%%%%%%%%%%%%%%%%%%%%%%%%%%%%%%%%%%%%%%%%%%%%%%%%%%%%%%%%%%%%%%%%%%%%%%%%%%%%%%%%%%%%%%%%%%%%%%%%%%%%%%%%%%%%%%%%%%%%%%%%%%%%%%%%%%%%%%%%%%%%%%%%%%%%%%%%%%%%

\section{Preliminaries}

This section is devoted to the definitions and some basic properties of Arf rings and generalized Gorenstein rings. Throughout this section, $R$ denotes a $d$-dimensional Cohen--Macaulay local ring with unique maximal ideal $\m$, residue field $k$ and canonical module $\rmK_R$. 

We start by recalling the definition of \emph{Ulrich} modules.

\begin{defn}[\cite{GK}]\label{def1.1} Let $M$ be a finitely generated $R$-module of dimension $s\ge0$,  and let $\fka$ be an $\m$-primary ideal of $R$. Then $M$ is said to be an Ulrich $R$-module with respect to $\fka$ provided the following conditions hold:
\begin{enumerate}[\rm(a)]
\item $M$ is a Cohen--Macaulay $R$-module.
\item $\rme_{\fka}^0(M) = \ell_R(M/\fka M)$.
\item $M/\fka M$ is a free $R/\fka$-module.
\end{enumerate}
Here $\ell_R(M)$ and $\rme_{\fka}^0(M)$ denote the length of $M$ as an $R$-module and the multiplicity of $M$ with respect to $\fka$, respectively. 
\end{defn}

Ulrich modules, with respect to the unique maximal ideal, were originally defined in \cite{BHU} as maximally generated maximal Cohen--Macaulay modules. This definition was then generalized in the article \cite{GOTWY} of the third author, Ozeki, Takahashi,  Watanabe, and Yoshida; in \cite{GOTWY}.
If $R$ is a non-regular and has minimal multiplicity, then it follows from the definition that $\m$ is an Ulrich ideal. Recall that $R$ is said to have {\it minimal multiplicity} if $\rme(R) = \ed(R) - \dim A + 1$, where $\ed(R)$ denotes the embedding dimension of $R$. Moreover higher syzygies of Ulrich ideals are Ulrich modules; see \cite[3.2]{GOTWY}. We refer the reader to \cite{GK, GOTWY, GOTWY2, GTT} for further information on Ulrich modules. 

\begin{defn}[\cite{GK}]\label{def1.2} $R$ is said to be a {\it generalized Gorenstein} ring, if either $R$ is Gorenstein, or $R$ is not Gorenstein, but there exists an $\m$-primary ideal $\fka$ of $R$ and an exact sequence of $R$-modules 
$$0 \to R \xrightarrow{\varphi} \rmK_R \to C \to 0,$$
where $C$ is an Ulrich $R$-module with respect to $\fka$, and the induced homomorphism 
$$
R/\fka \otimes_R \varphi : R/\fka \to \rmK_R/\fka \rmK_R
$$ is injective.
If the latter case occurs, then $R$ is called a {\it generalized Gorenstein ring with respect to $\fka$}.
\end{defn}

\begin{defn}[\cite{GTT}]\label{def1.3} $R$ is said to be an {\it almost Gorenstein ring} if it is Gorenstein, or not Gorenstein but is a generalized Gorenstein ring with respect to $\m$.
\end{defn}

Next we record some preliminary results pertaining to \emph{Arf rings}.

%%%%%%%%%%%%%%%%%%%%%%%%%%%%%%%%%%%%%%%%%%%%%%%%%%%%%%%%%%%%%%%%%%%%%%%%%%%%%%%%%%%%%%%%%%%%%%%%%%%%%%%%%%%%%%%%%%%%%%%%%%%%%%%%%%%%

\subsection*{On Arf rings} For this subsection on Arf rings, $A$ denotes a commutative Noetherian \emph{semi-local} ring satisfying the following condition:\\
$(\sharp)$ $A_\mathcal{M}$ is a one-dimensional Cohen-Macaulay ring for each maximal ideal $\mathcal{M}$ of $A$.

Let $\calF_A$ be the set of ideals of $A$ that contain a non-zerodivisor on $A$. Then, for each $I \in \calF_A$, there is a filtration of endomorphism algebras as follows:
$$
A \subseteq I:_{\rmQ(A)}I \subseteq I^2:_{\rmQ(A)} I^2 \subseteq \cdots \subseteq I^n:_{\rmQ(A)}I^n \subseteq \cdots \subseteq\overline{A}.
$$
Here $\overline{A}$ and $Q(A)$ denote the integral closure and the total quotient ring of $A$,  respectively. We set
$$
A^I = \bigcup_{n\geq 1}[I^n:_{\rmQ(A)}I^n].
$$
The ring $A^I$, a module-finite extension over $A$, is called the {\it blowup of $A$ at $I$}.
Notice, if $a \in I$ is a reduction of $I$, then one has:
$$
A^I = A\left[\frac{I}{a}\right], \text{ where } \frac{I}{a} = \left\{\frac{x}{a} ~\middle|~ x \in I\right\} \subseteq \rmQ(R).
$$

\begin{defn} \label{st} An ideal $I \in \calF_A$ is called a {\it stable ideal} provided $A^I=I:_{\rmQ(A)}I$. 
\end{defn}

Note that, for each $I \in \calF_A$, $I^n$ is stable for some $n\geq 1$. Moreover, an ideal $I$ is stable if and only if $I^2=xI$ for some $x \in I$; see \cite{Lipman} for details.

\begin{defn}[\cite{Arf, Lipman}] \label{Cahit}
$A$ is called an \emph{Arf} ring provided every integrally closed ideal $I \in \calF_A$ is stable.
\end{defn}

\begin{notation}\label{notation} For each nonnegative integer $n$, we set:
\begin{eqnarray*}
A_n = \left\{
 \begin{array}{l}
  A \ \ \ \ \ \  \  \ \ \ \ \ \  \text{if} \ n = 0, \\
  A_{n-1}^{J(A_{n-1})}    \ \    \ \ \text{if} \ n \geq 1,
 \end{array}
\right.
\end{eqnarray*}
where $J(A_{n-1})$ stands for the Jacobson radical of the ring $A_{n-1}$.
\end{notation}

\noindent
Notice $A_1 = A^{J(A)}$ is a one-dimensional Noetherian semi-local ring which is a module-finite extension over $A$. Moreover, $A_1$ satisfies the condition $(\sharp)$, namely the localization $(A_1)_{\mathcal{N}}$ is a Cohen--Macaulay local ring of dimension one for each $\mathcal{N} \in \Max(A_1)$, i.e., for each maximal ideal $\mathcal{N}$ of $A_1$.

The following characterization of Arf rings plays an important role for our argument; see, for example, the proof of Proposition \ref{3.2} and that of Theorem \ref{intro1.2}. One can deduce it from the results of Lipman \cite{Lipman}, but we include here a short and a different proof for  the sake of completeness.

\begin{prop}[see {\cite[2.2]{Lipman}}]\label{2.2}
The following conditions are equivalent.
\begin{enumerate}[\rm(i)]
\item $A$ is an Arf ring.
\item $\ed((A_n)_\mathcal{M}) = \rme((A_n)_\mathcal{M})$ for each $n\geq 0$ and  maximal ideal $\mathcal{M}$ of $A_n$.
\end{enumerate}
\end{prop}

\begin{proof}
$(i) \Rightarrow (ii)$: Let $B=A^J$ be the blowup of $A$ at its Jacobson radical $J=J(A)$. By the induction argument, it suffices to prove $B$ is Arf and $\ed(A_\mathcal{M}) = \rme(A_\mathcal{M})$ for each $\mathcal{M} \in \Max(A)$. Since $\mathcal{M} \in \calF_A$ is integrally closed and $A$ is Arf, $\mathcal{M}$ is stable, i.e., there exists element $f \in \mathcal{M}$ satisfying $\mathcal{M}^2 = f\mathcal{M}$, which yields $\ed(A_\M)=\rme(A_\M)$. Let us make sure of the Arf property for $B$.
Note that the Jacobson radical $J=J(A)$ is an integrally closed open ideal of $A$. Thus we choose $x \in J$ with $J^2 = x J$. Therefore we have 
$$
B=A^J =A\left[\frac{J}{x}\right]= \frac{J}{x}.
$$
Let $L \in \calF_B$ be an integrally closed ideal of $B$ and we will show that $L$ is stable. By setting $I=xL$, we get 
$$
I \subseteq J \subseteq A
$$
so that $I$ is an open ideal of $A$, i.e., $I \in \calF_A$. It is straightforward to show that $I$ is integrally closed, whence $I^2 = \xi I$ for some $\xi \in I$. Hence 
$$
L^2 = \frac{\xi}{x}L \; \; \;\text{and} \;\;\; \frac{\xi}{x} \in L
$$
which imply that $B$ is Arf, as desired.

$(ii) \Rightarrow (i)$ Let $I \in \calF_A$ be an integrally closed ideal of $A$. We may assume that $I$ is a proper ideal of $A$. Localizing $A$ at its maximal ideal, we may also assume $A$ is a local ring with maximal ideal $\m$. Since $A$ has minimal multiplicity, there exists $x \in \m$ such that $\m^2 = x \m$. Let 
$$
B := A^{\m} =A\left[\frac{\m}{x}\right]= \frac{\m}{x}
$$ 
be the blowup of $A$ at $\m$. As $I \subseteq \m$, we see that $L:=\frac{I}{x} \subseteq B$ is an ideal of $B$. One can show that $L$ is integrally closed. We now assume that $L$ is not stable in $B$. Then $L \subsetneq B$ and, for each $\mathcal{N} \in \Max(B)$, one has:
$$
\ell_{B_\mathcal{N}}(B_\mathcal{N}/LB_\mathcal{N}) \le \ell_B(B/L) \le \ell_A(B/L) < \ell_A(A/I).
$$
Repeating the same process for $B_\mathcal{N}$ recursively contradicts the fact that $\ell_A(A/I)$ is finite. Therefore $L$ is stable, so that we can choose $\xi \in L$ satisfying $L^2 = \xi L$. Since $I =xL$, we have 
$$
I^2 = (x\xi)I,
$$
which yields that $A$ is an Arf ring.
\end{proof}
%\todo[inline]{Edit the proof of \ref{2.2}}

We finish this section with a few more observations on Arf property.

\begin{prop}\label{3.2}
Assume $A$ is one-dimensional, Cohen--Macaulay and local with unique maximal ideal $\m$. Let $C$ be an intermediate ring between $A$ and $\rmQ(A)$ such that $C$ is a finitely generated $A$-module. Assume $\rme(C_\mathcal{M}) \le 2$ for each maximal ideal $\mathcal{M}$ of $C$. If $C \subseteq D \subseteq \rmQ(C)$ is an intermediate ring that is a finitely generated $C$-module, then $\e(D_{\mathcal{N}})\leq 2$ for each maximal ideal $\mathcal{N}$ of $D$. In particular $C$ is an Arf ring.
\end{prop}

\begin{proof} Let $C \subseteq D \subseteq \rmQ(C)$ be an intermediate ring such that $D$ is a module-finite extension over $C$. Let $\mathcal{N} \in \Max(D)$, and set $\mathcal{M}=\mathcal{N}\cap C$. Then $\mathcal{M}$ is a maximal ideal of $C$. Note that  $D_\mathcal{M}\cong D_{\mathcal{N}}$ and $\rmQ(C)_{\mathcal{M}}\cong \rmQ(C_{\mathcal{M}})$. Therefore,
$
C_{\M} \subseteq D_{\M} \subseteq \rmQ(C_{\M}),
$
where $D_{\mathcal{M}}$ is a module-finite extension over $C_\mathcal{M}$. Since $\e(C_\mathcal{M}) \le 2$, \cite[12.2]{GTT} shows $\e(D_\mathcal{N}) \le 2$. 

Now, to see $C$ is Arf, let $D=C_{n}$, a blowup of $C$; see Notation \ref{notation}. Then, $\e(D_\mathcal{N}) \le 2$, which implies $1\leq \ed(D_\mathcal{N}) = \e(D_\mathcal{N})$. Therefore, $C$ is an Arf ring by Proposition \ref{2.2}.
\end{proof}

%%%%%%%%%%%%%%%%%%%%%%%%%%%%%%%%%%%%%%%%%%%%%%%%%%%%%%%%%%%%%%%%%%%%%%%%%%%%%%%%%%%%%%%%%%%%%%%%%%%%%%%%%%%%%%%%%%%%%%%%%%%%%%%%%%%%%%%%%%%%%%%%%%%%%%%%%%%%%%%%%%%%%%%%%%%%%%%%%%%%%%%%%%%%%%%%%%%%%%%%%%%%%%%%%%%%%%%%%%%%%%%%%%%%%%%%%%%%%%%%%%%%%%%%%%%%%%%%%%%%%%%%%%%%%%%%%%%%%%%%%%%%%%%%%%%%%%%%%%%%%%%%%%%%%%%%%%%%%%%%%%%%%%%%%%%%%%%%%%%%%%%%%%%%%%%%%%%%%%%%%%%%%%%%%%%%%%%%%%%%%%%%%%%%%%%%%%%%%%%%%%%%%%%%%%%%%%%%%%%%%%%%%%%%%%%%%%%%%%%%%%%%%%%%%%%%%%%%%%%%%%%%%%%%%%%%%%%%%%%%%%%%%%%%%%%%%%%%%%%%%%%%%%%%%%%%%%%%%%%%%%%%%%%%

\section{A Proof of the main result}

\begin{setup}\label{2.3}
Let $(R, \m)$ be a one-dimensional Cohen--Macaulay local ring with canonical module $\rmK_R$. We set $B = \m:_{\rmQ(R)}\m$, the endomorphism algebra of the maximal ideal $\m$, where $\rmQ(R)$ denotes the total quotient ring of $R$.

We fix a canonical ideal $I$ of $R$, and assume $I$ contains a parameter ideal $aR$ of $R$ as a reduction. Let 
$$
S =R[K], \text{ where } K = \frac{I}{a} = \left\{\frac{x}{a} ~\middle|~ x \in I\right\} \subseteq \rmQ(R).
$$ 
Furthermore we define the conductor as $\fkc = R:_{\rmQ(R)}S$. Notice $R \subseteq K \subseteq \overline{R}$, where $\overline{R}$ is the integral closure of $R$ in $\rmQ(R)$. 

Note $K \cong I \cong \rmK_R$, and both $B$ and $S$, being module finite extensions of $R$, are one-dimensional semi-local rings. Note also that $\e(R)=\e^{0}_{\m}(S)=\e^{0}_{\m}(B)$.
\pushQED{\qed} 
\qedhere
\popQED
\end{setup}

We will make use of the following theorems for our proof of Theorem \ref{3.1}, which is the key ingredient in our proof of Theorem \ref{intro1.2}. 

\begin{thm}[{\cite[5.1]{GMP}}]\label{2.4}
The following conditions are equivalent.
\begin{enumerate}[\rm(i)]
\item $R$ is an almost Gorenstein ring and $\ed(R) =\e(R)$.
\item $B$ is a Gorenstein ring.
\end{enumerate}
\end{thm}

\begin{thm}[{\cite[4.18]{GK}}]\label{2.5}
Assume there is an element $x \in \m$ such that $\m^2=x\m$. Then the following conditions are equivalent.
\begin{enumerate}[\rm(i)]
\item $R$ is a generalized Gorenstein ring, but $R$ is not an almost Gorenstein ring.
\item $B$ is not a Gorenstein ring, but $B$ is a generalized Gorenstein local ring with maximal ideal $\n$ such that $\n^2=x \n$.
\end{enumerate}

When one of these equivalent conditions hold, we have $R/\m \cong B/\n$, and 
$$
\ell_B(B/(B:B[L]))=\ell_R(R/\fkc) -1,
$$
where $L=BK$.
\end{thm}

%Next is a very special case of a result of Hayasaka and Hyry \cite{}; this result will be used for the proofs of Theorem \ref{3.1} and Corollary \ref{77}.

%\begin{thm}\label{HH} (Hayasaka and Hyry \cite[1.1]{}) If $\q$ is a parameter ideal of $R$, then $\ell_R(R/\q) \ge  \rme^0_{\q}(R)$.
%\end{thm}

\begin{thm}\label{3.1}
Assume $\rme(R) \ge 3$. Then the following conditions are equivalent.
\begin{enumerate}[\rm(i)]
\item $R$ is a generalized Gorenstein ring with minimal multiplicity.
\item $S$ is Gorenstein, and there is an integer $N>0$ such that the following hold:
\begin{enumerate}[\rm(a)]
\item $S=R_N$.
\item For each integer $n=0, \ldots, N-1$, it follows $R_n$ is a local ring such that $\ed(R_n) = \rme(R_n) = \rme(R)$.
\end{enumerate}
\end{enumerate}
Furthermore, if condition (ii) holds, then we have $N = \ell_R(R/\fkc)$.
\end{thm}

\begin{proof}
We set $\ell=\ell_R(R/\fkc)$. By \cite[3.5, 3.7]{GMP}, let us remark that $R$ is Gorenstein if and only if $\ell=0$. Besides, by \cite[3.5, 3.16]{GMP}, $R$ is a non-Gorenstein almost Gorenstein ring is equivalent to $\ell=1$.

$(i) \Rightarrow (ii)$ Since $\ed(R) = \rme(R)$, we choose $x \in \m$ such that $\m^2 = x \m$. Set $N=\ell = \ell_R(R/\fkc)$. 
As $\rme(R) \ge 3$, $R$ is not a Gorenstein ring, so that $\ell>0$. 
If $\ell=1$, then $R$ is an almost Gorenstein ring and $S=\m:_{\rmQ(R)}\m$ is a Gorenstein ring by Theorem \ref{2.4}. Suppose that $\ell > 1$ and the assertion holds for $\ell -1$. We then have $R$ is not an almost Gorenstein ring. By Theorem \ref{2.5}, $B$ is a generalized Gorenstein ring, but not Gorenstein. Furthermore, $\n^2 = x \n$ and $R/\m \cong B/\n$, where $\n$ denotes the maximal ideal of $B$. 
Hence we have:
$$
\rme(R) = \rme^0_{\m}(B)=\ell_R(B/xB)=\ell_B(B/xB) = \rme^0_{\n}(B).
$$
Note that $L=KB$ is a $B$-submodule of $\rmQ(B)$ such that $B \subseteq L \subseteq \overline{B}$, $L \cong \rmK_B$, and $S=B[L]$; see \cite[5.1]{CGKM} for the details. Therefore, we get
$$
\ell_B(B/\fka) = \ell_R(R/\fka) -1 = \ell -1,
$$
where $\fka = B:_{\rmQ(R)}B[L]$ denotes the conductor of $B[L]$. Hence, by induction hypothesis, $S = B[L]$ is a Gorenstein ring and $R_n=B_{n-1}$ is a local ring with minimal multiplicity $\rme(R)$ for every $1 \le n < \ell$. 

$(ii) \Rightarrow (i)$ If $N=1$, then $S=\m:_{\rmQ(R)}\m$, since $\m$ is stable. Hence, by Theorem \ref{2.4}, the Gorensteinness of $S$ implies that $R$ is an almost Gorenstein ring and $\ell=N$. Suppose that $N>1$ and the assertion holds for $N-1$. Since $R=R_0$ has minimal multiplicity, there exists $x \in \m$ such that $\m^2 = x\m$. Then, since $N \ge 2$, it follows 
$$
R_1 = \m:_{\rmQ(R)}\m=B
$$
is a local ring with minimal multiplicity $\rme(R)$. Note $xB$ is a parameter ideal of $B$. Moreover we have
$$
\rme(R) = \rme^0_{\m}(R) = \rme^0_{\m}(B) = \ell_R(B/x B) \ge \ell_B(B/x B) \ge \rme^0_{\n}(B) =\rme(R),
$$
where $\n$ denotes the maximal ideal of $B$. It follows that $R/\m \cong B/\n$. Thus, by \cite{Rees}, $x B$ is a reduction of $\n$, whence $\n^2 = x \n$. Therefore the induction arguments shows $B$ is a generalized Gorenstein ring. Since $B$ is not a Gorenstein ring, by Theorem \ref{2.5}, we see $R$ is a generalized Gorenstein ring, but not an almost Gorenstein ring, and also
$$
\ell_B(B/\fka) = \ell -1,
$$
where $\fka = B:B[KB]$. Note that $R_n$ has minimal multiplicity $\rme(R)$ for every $1 \le n <N$, so is $B_{n-1}$. By the induction hypothesis, we conclude that $N-1 = \ell -1$, as desired.
\end{proof}

We are now ready to prove our main result, namely Theorem \ref{intro1.2} advertised in the introduction.

\begin{proof}[Proof of Theorem \ref{intro1.2}] Notice, if $R$ is Arf, then the maximal ideal $\m$ is stable so that $R$ has minimal multiplicity; see Definition \ref{st}. Therefore, throughout, we may assume $R$ has minimal multiplicity. We may also assume $\e(R)\ge 3$ by Proposition \ref{3.2}.
So, by Theorem \ref{3.1}, we conclude $S$ is Gorenstein, $S=R_N$ for some positive integer $N$, and
the blowup $R_n$ of $R$ is a local ring with minimal multiplicity for each integer $n=0, \ldots, N-1$.
 
$(i) \Rightarrow (ii)$: Assume $R$ is Arf. Then, for each nonnegative integer $n$ and each maximal ideal $\mathcal{M}$ of $R_n$, it follows from Proposition \ref{2.2} that $\ed((R_n)_\mathcal{M}) = \rme((R_n)_\mathcal{M})$. Since $S_v=R_{N+v}$ for each nonnegative integer $v$, we conclude, using Proposition \ref{2.2} once more, that $S$ is Arf. One can now observe that $S_{\mathcal{M}}$ has minimal multiplicity for all maximal ideals $\mathcal{M}$ of $S$. Thus the Cohen-Macaulay type of $S_{\mathcal{M}}$ equals $\e(S_{\mathcal{M}})-1$, provided $S_{\mathcal{M}}$ is not regular. In particular, since $S$ is Gorenstein, we obtain $e(S_{\mathcal{M}})\leq 2$.

$(ii) \Rightarrow (i)$: Assume $R$ has minimal multiplicity and $\e(S_{\mathcal{M}})\leq 2$ for each maximal ideal $\mathcal{M}$ of $S$. Letting $C=S$, we see from Proposition \ref{3.2} that $S$ is an Arf ring. Letting $A=S$ in Proposition \ref{2.2}, we conclude that all the localizations of the blowups of $S$ have minimal multiplicity, i.e., for each nonnegative integer $l$, and for each maximal ideal $\mathcal{N}$ of $S_{l}$, the ring $(S_{l})_{\mathcal{N}}$ has minimal multiplicity. Since $S=R_{N}$, this property is also true for the blowups of $R$, which are local rings. Therefore $\ed(R_n)=\e(R_n)$ for each integer $n=0, \ldots, N-1$. Finally we deduce from Proposition \ref{2.2} that $R$ is Arf.
\end{proof}

%%%%%%%%%%%%%%%%%%%%%%%%%%%%%%%%%%%%%%%%%%%%%%%%%%%%%%%%%%%%%%%%%%%%%%%%%%%%%%%%%%%%%%%%%%%%%%%%%%%%%%%%%%%%%%%%%%%%%%%%%%%%%%%%%%%%%%%%%%%%%%%%%%%%%%%%%%%%%%%%%%%%%%%%%%%%%%%%%%%%%%%%%%%%%%%%%%%%%%%%%%%%%%%%%%%%%%%%%%%%%%%%%%%%%%%%%%%%%%%%%%%%%%%%%%%%%%%%%%%%%%%%%%%%%%%%%%%%%%%%%%%%%%%%%%%%%%%%%%%%%%%%%%%%%%%%%%%%%%%%%%%%%%%%%%%%%%%%%%%%%%%%%%%%%%%%%%%%%%%%%%%%%%%%%%%%%%%%%%%%%%%%%%%%%%%%%%%%%%%%%%%%%%%%%%%%%%%%%%%%

\section{Corollaries of the main argument}

In this section we maintain the notations of Setup \ref{2.3}. We give applications of our argument and  obtain new characterizations of Arf rings. In particular, we extend a result of Barucci and Fr\"oberg \cite{BF} and determine certain conditions that make the idealization $R \ltimes \fkc$ to be a generalized Gorenstein Arf ring; see Corollaries \ref{3.4}, \ref{3.7} and Theorem \ref{4.1}.

We start by giving two examples that show Arf and generalized Gorenstein properties are independent of each other, in general.

\begin{ex} \label{ornek1} Let $k$ be a field and set $R= k[\![t^3,t^7,t^{11}]\!]$. Then $R$ is a Cohen-Macaulay, non-Gorenstein ring of Cohen-Macaulay type two. We proceed and prove that $R$ is almost Gorenstein, and hence generalized Gorenstein, but is not Arf; see Definitions \ref{def1.3} and \ref{Cahit}.

Note that $K_R=R+Rt^4$. As $\m K_R=\m t^4 \subseteq R$, we conclude from \cite[3.11]{GMP} that $R$ is almost Gorenstein.

To show $R$ is not Arf, we compute the blowup of $R$ at $\m$:
\begin{equation}\tag{\ref{ornek1}.1}
R_1=R^{\m}=R\left[ \frac{\m}{t^3}\right]=k[\![t^3,t^4]\!]
\end{equation}
Here, in (\ref{ornek1}.1), the second equality holds since $t^3$ is a reduction of $\m$; see the discussion preceding Definition \ref{st}. As $R_1$ does not have minimal multiplicity, $R$ is not Arf by Proposition \ref{2.2}.
\end{ex}

\begin{ex}\label{ornek2} Let $k$ be a field and set $R= k[\![t^4,t^7,t^9,t^{10}]\!]$. Then $R$ is not generalized Gorenstein ring; see \cite[4.27]{GK}. To see $R$ is Arf, we proceed as in Example \ref{ornek1}.
\begin{equation}\tag{\ref{ornek2}.1}
R_1=R^{\m}=R\left[ \frac{\m}{t^4}\right]=k[\![t^3,t^4, t^5]\!]
\end{equation}
Here, in (\ref{ornek2}.1), the second equality holds since $t^4$ is a reduction of $\m$. Letting $\mathfrak{m_1}$ to be the unique maximal ideal of $R_1$, we get:
\begin{equation}\tag{\ref{ornek2}.2}
R_2=R_1^{\m}=R\left[ \frac{\mathfrak{m_1}}{t^4}\right]=k[\![t]\!].
\end{equation}
The second equality in (\ref{ornek2}.2) holds since $\mathfrak{m_1}^2=t^3\mathfrak{m_1}$. As each blowup of $R$ is contained in the integral closure $\overline{R}=k[\![t]\!]$ of $R$, we conclude that $R=R_{0}$, $R_1$ and $R_2=\overline{R}$ are the only distinct blowups of $R$. Since each of these blowups have minimal multiplicities, we see that $R$ is Arf by Proposition \ref{2.2}. $\qed$
\end{ex} 

Note that it follows from the definition that $S$ is a local ring in case $R$ is a numerical semigroup ring.
Hence the ring $S$ in Example \ref{ornek1} is local, whilst $R$ is a non-Arf ring with minimal multiplicity three. The next corollary of Theorem \ref{intro1.2} shows that such a ring $R$ must be Arf in case $S$ is not local.

\begin{cor} \label{77} Assume $\ed(R)=\e(R)=3$, i.e., $R$ has minimal multiplicity three. If $S$ is not a local ring, then $R$ is an Arf ring. 
\end{cor}

\begin{proof} Note that, since $R$ has minimal multiplicity, there exists an element $x \in \m$ such that $\m^2 = x \m$. We conclude by \cite[4.8]{GK} that $R$ is a generalized Gorenstein ring. Therefore we have:
\begin{align}\notag{}
3  = \e(R) =   \e^0_{\m}(S) = \e^0_{xR}(S) = \ell_R(S/ x S)  & \ge  \ell_S(S/ x S)  &  \\& =  \sum_{\M \in \Max(S)} \ell_{S_{\M}}(S_{\M}/x S_{\M}) \notag{} \\&  \ge \sum_{\M \in \Max(S)} \e(S_{\M}), \notag{}
\end{align}
where $\Max(S)$ denotes the set of all maximal ideals of $S$. %Here the second inequality of (\ref{77}.1) is due to Theorem \ref{HH}. 
Since $S$ is not local, there are at least two distinct ideals in $\Max(S)$.
This implies that $\e(S_{\M}) \le 2$ for each $\mathcal{M} \in \Max(S)$. Now, since $R$ is a generalized Gorenstein ring with minimal multiplicity, it follows from Theorem \ref{intro1.2} that $R$ is Arf.
\end{proof}

Next we recall a beautiful result of Barucci and Fr\"oberg \cite[13]{BF} that gives a characterization of  almost Gorenstein Arf numerical semigroup rings.

%Let $V=k[[t]]$ be the  formal power series ring over a field $k$. We set
\begin{thm}[see {\cite[13]{BF}}]\label{3.5}
Let $\ell$ be a positive integer and $0 < a_1 < a_2 < \cdots < a_{\ell}$ be integers such that $\gcd(a_1, a_2, \ldots, a_{\ell})=1$. Let $k$ be a field, $R=k[\![t^{a_1}, t^{a_2}, \ldots, t^{a_\ell}]\!]$ be the numerical semigroup ring over $k$, and let $H=\left<a_1, a_2, \ldots, a_{\ell}\right>$ be the corresponding semigroup.
Then the following are equivalent.
\begin{enumerate}[\rm(i)]
\item $R$ is an almost Gorenstein Arf ring.
\item $2+a_i \in H$ for each $i=1, \ldots, \ell$. 
\end{enumerate}
\end{thm}

In Corollary \ref{3.4} and Proposition \ref{3.7}, we will obtain natural generalizations of Theorem \ref{3.5}. These results will be useful to construct new examples of almost Gorenstein Arf rings; see Examples \ref{ornek3} and \ref{ornek4}. %Note that, if $R$ is almost Gorensein, but not Gorenstein, then $S = B=\m:_{\rmQ(R)}\m$; see \cite[3.16]{GMP}.

Recall that, in the following, and unless otherwise stated, we maintain the notations of Setup \ref{2.3}.

\begin{cor}\label{3.4} The following conditions are equivalent.
\begin{enumerate}[\rm(i)]
\item $R$ is an almost Gorenstein Arf ring.
\item $\e(B_{\N}) \le 2$ for each maximal ideal $\N$ of $B$. 
\end{enumerate}
\end{cor}

\begin{proof} First assume $\e(R)\leq 2$. Then  it follows from Proposition \ref{3.2} that $R$ is Arf. Hence, since $R$ is Gorenstein, (i) follows. If $\e(R)=1$, then $R$ is regular so that $B=R=\overline{R}$; in particular (ii) holds. Moreover, if $\e(R)=2$, then (ii) follows from \cite[12.2]{GTT}. Consequently we may assume $\e(R)\geq 3$.

$(i) \Rightarrow (ii)$ Assume $R$ is Arf. Then $R$ has minimal multiplicity so that $\m$ is stable and $B=R_1=R^{\m}$; see Definition \ref{st} and Notation \ref{notation}. Therefore, by Proposition \ref{2.2}, $B_{\N}$ has minimal multiplicity for each maximal ideal $\N$ of $B$. This implies that $B$ is Gorenstein. Consequently, given a maximal ideal $\N$ of $B$, since $B_{\N}$ is a Gorenstein ring with minimal multiplicity, we conclude that $\e(B_{\N})\leq 2$.

$(ii) \Rightarrow (i)$ Assume $\e(B_{\N}) \le 2$ for each maximal ideal $\N$ of $B$. Then $B$ is Gorenstein. Hence, by Theorem \ref{2.4}, $R$ is an almost Gorenstein ring with minimal multiplicity.

Now, if $R$ is Gorenstein, then $R$ is a hypersurface with $\e(R)\leq 2$. In particular, $R$ is Arf by Proposition \ref{3.2}. If $R$ is not Gorenstein, since it is almost Gorenstein, we have from \cite[3.16]{GMP} that $B=S$. Hence, by the hypothesis, $\e(S_{\N})\leq 2$ for each maximal ideal $\N$ of $S$. So $R$ is Arf by Theorem \ref{intro1.2}.
\end{proof}

In passing, we give a short proof of Theorem \ref{3.5} which is different from the argument of Barucci and Fr\"oberg \cite{BF}. We will use Corollary \ref{3.4} and the fact that, if $R$ is a numerical semigroup ring as in Theorem \ref{3.5}, then $B=\m:_{\rmQ(R)}\m$ is also a numerical semigroup ring (in particular, $B$ is local); see, for example, \cite[2.13]{GS}.

%\begin{lem}(\cite[see ??]{CGKM}) \label{num} Assume $R$ is a numerical semigroup ring as in Theorem \ref{3.5}. Then $B=\m:_{\rmQ(R)}\m$ is also a numerical semigroup ring. In particular $B$ is local.
%\end{lem}

\begin{proof} [A proof of Theorem \ref{3.5}] Assume $R$ is an almost Gorenstein Arf ring. Then it follows from Corollary \ref{3.4} that $\e(B)\leq 2$. Therefore $t^2\in B$ and $t^2\m \subseteq \m$. This shows $2+a_i \in H$ for each $i=1, \ldots, \ell$.

Now assume $2+a_i \in H$ for each $i=1, \ldots, \ell$. Hence $t^2\m \subseteq \m$ and $t^2\in B$. Let $V$ and $\n$ denote $\overline{R}$ and the unique maximal ideal of $B$, respectively. Note that $\n V=xV$ for some $x\in \n$. Then, seeting $\q=(x)$, we have:
\begin{equation}\notag{}
\e(B)=e^0_{\n}(B)=\e^0_{\n}(V)=\e_{\q}^0(V)= \ell_B(V/\q V)= \ell_B(V/\n V).
\end{equation}
Here the third equality holds since $\q$ is a reduction of $\n$. Therefore it follows that:
\begin{equation}\notag{}
\e(B)=\ell_B(V/\n V)\leq \ell_B(V/t^2 V)=\ell_V(V/t^2 V)=2.
\end{equation}
Consequently, by Corollary \ref{3.4}, $R$ is almost Gorenstein and Arf.
\end{proof}

Our next result yields an extension of Theorem \ref{3.5}. 

\begin{prop}\label{3.7}
Assume $R$ is a numerical semigroup ring as in Theorem \ref{3.5}.  Assume further $R$ is a generalized Gorenstein ring. Then the following conditions are equivalent.
\begin{enumerate}[\rm(i)]
\item $R$ is an Arf ring.
\item $R$ has minimal multiplicity, $2+ \ell_R(R/\fkc)\cdot a_1 \in H$, and $2 + a_i \in H$ for each $i=2, \ldots, \ell$. 
\end{enumerate}
\end{prop}

\begin{proof} Recall that an Arf ring has minimal multiplicity. Hence we may assume $R$ has minimal multiplicity throughout the proof.

Assume $a_1=\e(R)\leq 2$. Then $R$ is Arf by Proposition \ref{3.2}. Hence (i) holds. Since $a_1\in H$, we see that $2+a_i\in H$ for all $i=1, \ldots, \ell$. Moreover, as $R$ is Gorenstein, $\ell_R(R/\fkc)=0$  so that $2+ \ell_R(R/\fkc)\cdot a_1 \in H$. In particular (ii) holds. Consequently we may assume $\e(R)\geq 3$. So it follows from Theorem \ref{3.1} that $S=R_{N}$ is Gorenstein and $R_i$ is a local ring with minimal multiplicity $\rme(R)$ for each $i=0, \ldots, N-1$, where $N= \ell_R(R/\fkc)$. 

Fix $i$ with $1\leq i \leq N-1$, and let $\m_{i}$ denote the unique maximal ideal of the local ring $R_{i}$. Then, since $\m_{i}$ is stable, we have $R_{i}=\m_{i-1}:_{\rmQ(R_i)}\m_{i-1}$; see Definition \ref{st}. (Here $\m_0=\m$).

Note, by Proposition \ref{2.2}, $R$ is Arf if and only if $R_{N-1}$ is Arf. Also, since $S=R_{N}=\m_{N-1}:_{\rmQ(R_N)}\m_{N-1}$ is Gorenstein, we see from Theorem \ref{2.4} that $R_{N-1}$ is almost Gorenstein. So $R$ is Arf if and only if $R_{N-1}$ is almost Gorenstein Arf.

The ring $R_{i}$, since it has minimal multiplicity, is a numerical semigroup ring; see \cite[Section 6]{CGKM}. More precisely, $R_i=k[\![H_{i}]\!]$, where $H_i$ is the numerical semigroup generated by $\{a_1, a_2-ia_1, \ldots, a_{\ell}-ia_1\}$. Therefore, by Theorem \ref{3.5}, for each $j=2, \ldots, N$,  we have:
\begin{equation}\tag{\ref{3.7}.1}
R \text{ is Arf if and only if } 2+a_1\in H_{N-1} \text{ and } a_j-(N-1) \cdot a_1+2 \in H_{N-1}.
\end{equation}

Now we claim that $x\in H_1$ if and only if $x+a_1\in H$ for each nonnegative integer $x$. First we proceed by assuming the claim.

It follows from the claim, for a nonnegative integer $x$, we have that: 
\begin{equation}\tag{\ref{3.7}.2}
x\in H_{N-1} \text{ if and only if } x+(N-1)a_1\in H.
\end{equation}
Now, letting $x=a_1+2$, we obtain from (\ref{3.7}.2) that:
\begin{equation}\tag{\ref{3.7}.3}
2+a_1\in H_{N-1} \text{ if and only if } 2+a_1 \cdot N=2+ \ell_R(R/\fkc)\cdot a_1  \in H
\end{equation}
Moreover, for each $j=2, \ldots N$, it follows from (\ref{3.7}.2) that:
\begin{equation}\tag{\ref{3.7}.4}
a_j-(N-1) \cdot a_1+2 \in H_{N-1} \text{ if and only if } [ a_j-(N-1) \cdot a_1+2 \in H_{N-1} ]+(N-1) \cdot a_1 \in H.
\end{equation}
Therefore, by (\ref{3.7}.1), (\ref{3.7}.3) and (\ref{3.7}.4), we conclude that $R$ is Arf if and only if
$2+ \ell_R(R/\fkc)\cdot a_1 \in H$, and $2 + a_j \in H$ for each $j=2, \ldots, \ell$. 

Next, to complete the proof, we justfiy the above claim. Let $x\in H_1$. Then $t^{x} \in  R_{1}=\m_{}:_{\rmQ(R)}\m_{}$, i.e., $t^{x}\m \subseteq \m$. This implies $t^{x+a_1} \in \m$, or equivalently, $x+a_1 \in H$. 

Now assume $x+a_1\in H$ for each nonnegative integer $x$. Since $R$ has minimal multiplicity, it follows that $\m^2=t^{a_1}\m$. Therefore,
\begin{equation}\tag{\ref{3.7}.5}
R_1=R^{\m}=R\left[ \frac{\m}{t^a_1}\right]=\bigcup_{i \geq 0} \frac{\m^{i}}{(t^{a_1})^{i}}= \frac{\m}{t^{a_1}}
\end{equation}
Here, in (\ref{3.7}.5), the second equality holds since $t^{a_1}$ is a reduction of $\m$; see the discussion preceding Definition \ref{st}. Moreover the fourth equality is due to the fact that 
$\displaystyle{\frac{\m^i} {(t^{a_1})^i}=\frac{\m}{t^{a_1}}}$ for each $i\geq 2$. Hence, since $x+a_1 \in H$, we have:
$$\displaystyle{t^x=\frac{t^{x+a_1}}{t^{a_1}} \in \frac{\m}{t^{a_1}}=R_1=\m:_{\rmQ(R)}\m}$$ This implies that $x\in H_1$.
\end{proof}

We are now ready to construct several examples. The rings in Examples \ref{ornek3} and \ref{ornek4} are almost Gorenstein Arf rings. On the other hand, the one in Example \ref{ornek5} is a generalized Gorenstein Arf ring which is not almost Gorenstein.

\begin{ex} \label{ornek3} Let $k$ be a field, $e\geq 2$ an integer, and $R=k[\![t^e, t^{e+1}, \ldots, t^{2e-1}]\!]$ be the numerical semigroup ring. Since the conductor of the corresponding semigroup is $e$, we see $(e+i)+2 \in H$ for all $i=0, \ldots, e-1$. Therefore, by Theorem \ref{3.5}, $R$ is an almost Gorenstein Arf ring. 

Notice $\m^2=t^e\m$. So it follows that $$B=\m:_{\rmQ(R)}\m=R\left[ \frac{\m}{t^e}\right]=k[\![t,t^2, \ldots, t^{e-1}]\!]=k[\![t]\!]=\overline{R}.$$ Hence $\e(B)=1\leq 2$, cf., Theorem \ref{3.4}.  
\qedhere
\pushQED{\qed} 
\popQED	
\end{ex}
%$V = k[[t]]$ the formal power series ring. $R=k[\![t^{e + i} \mid i=0, \ldots, e-1]\!]
%Let $e \ge 2$ be an integer. Then $R = k[[t^{e + i} \mid 0 \le i \le e-1]]$ is an almost Gorenstein Arf ring with $\m:\m = V$. 

\begin{ex}\label{ornek4} Let $k$ be a field, $e\geq 3$ an integer, and $R=k[\![t^e, t^{e+2}, \ldots, t^{2e-1}, t^{2e+1}]\!]$ be the numerical semigroup ring. Since the conductor of the corresponding semigroup is $e+2$, we see that $2+e$, $2+(2e+1)$, as well as $2+(e+i)$, for all $i=2, \ldots, e-1$, belong to $H$. Therefore, by Theorem \ref{3.5}, $R$ is an almost Gorenstein Arf ring. 

Notice $\m^2=t^e\m$. So it follows that $$B=\m:_{\rmQ(R)}\m=R\left[ \frac{\m}{t^e}\right]=k[\![1, t^2, t^3, \ldots, t^{e-1},t^{e+1}]\!]=k[\![t^2,t^3]\!].$$ Hence $\e(B)=2\leq 2$, cf., Theorem \ref{3.4}.  
\qedhere
\pushQED{\qed} 
\popQED	
\end{ex}

\begin{ex} \label{ornek5} Let $k$ be a field and let $R=k[\![t^5, t^{16}, t^{17}, t^{18}, t^{19}]\!]$ be the numerical semigroup ring. It was proved in \cite[4.27]{GK} that $R$ is a generalized Gorenstein ring.

The canonical module $\rmK_R$ of $R$ is $R+Rt+Rt^2+Rt^3$. As $t\in \rmK_R$, we have that $k[\![t]\!]\subseteq R[\rmK_R]=S$. Since the conductor of the corresponding semigroup $H$ is $15$, setting $V=\overline{R}$, we see that $\fkc = R:_{\rmQ(R)}S=(tV)^{15}$. Therefore we have $\ell_R(R/\fkc)=\ell_R(V/\fkc)-\ell_R(V/R)=15-12=3$. This implies $2+\ell_R(R/\fkc)\cdot a_1 =17 \in H$. Moreover $2 + a_i \in\{18, 19, 20, 21\} \subseteq H$ for each $i=2, \ldots, 5$. So $R$ is Arf by Proposition \ref{3.7}. Note $\m \rmK_R \nsubseteq R$ since $t^5 \cdot t \in \m \rmK_R -R$. Consequently, $R$ is not almost Gorenstein; see \cite[3.11]{GMP}.
\qedhere
\pushQED{\qed} 
\popQED	
\end{ex}

Next we move to another application somewhat different in nature.

%%%%%%%%%%%%%%%%%%%%%%%%%%%%%%%%%%%%%%%%%%%%%%%%%%%%%%%%%%%%%%%%%%%%%%%%%%%%%%%%%%%%%%%%%%%%%%%%%%%%%%%%%%%%%%%%%%%%%%%%%%%%%%%%%%%%%%%%%%%%%%%%%%%%%%%%%%%%%%%%%%%%%%%%%%%%%%Almost Gorenstein Arf rings obtained by idealization
%%%%%%%%%%%%%%%%%%%%%%%%%%%%%%%%%%%%%%%%%%%%%%%%%%%%%%%%%%%%%%%%%%%%%%%%%%%%%%%%%%%%%%%%%%%%%%%%%%%%%%%%%%%%%%%%%%%%%%%%%%%%

\subsection*{An application of Theorem \ref{intro1.2} on a certain idealization} 
The aim of this subsection is to give a necessary and sufficient condition for the ring $A = R \ltimes \fkc$ to be generalized Gorenstein and Arf, where $A = R \ltimes \fkc$ is the idealization of $\fkc = R:_{\rmQ(R)}S$ (Recall we follow the notations of Setup \ref{2.3}). Note that, when $R$ is a generalized Gorenstein, $A$ is always generalized Gorenstein; see \cite[4.15]{GK}. %By \cite[Theorem 4.15]{GK}, $A$ is a generalized Gorenstein ring. 

\begin{thm}\label{4.1} Assume $R$ is a generalized Gorenstein ring and set $A = R \ltimes \fkc$, where  $\fkc = R:_{\rmQ(R)}S$.  Then the following conditions are equivalent.
\begin{enumerate}[\rm(i)]
\item $A$  is an Arf ring.
\item $R$ has minimal multiplicity and $S = \overline{R}$.
\end{enumerate}
Moreover, if (i) or (ii) holds, then $R$ is an Arf ring. 
\end{thm}

\begin{proof} We start by noting that $K_A \cong \Hom_R(\fkc, K) \times K \cong S \times K$; see \cite[Section 4]{GK}. Set $L= S \times K$. Then it follows that $A[L] = L^2 = S \times S$; see \cite[4.14]{GK}. Notice $A$ is a generalized Gorenstein ring; this follows from \cite[4.15]{GK} in case $A$ is not Gorenstein. Note also that $\n=\m \times \fkc$ is the unique maximal ideal of $A$

$(1) \Rightarrow (2)$ Assume $A$ is an Arf ring. Then, by Theorem \ref{intro1.2}, $\e(A[L]_{\mathcal{N}})\leq 2$ for each maximal ideal $\mathcal{N}$ of $A[L]$. Thus $2 \cdot \e(S_{\M})=\e(S_{\M} \times S_{\M})\leq 2$, i.e., $\e(S_{\M})\leq 1$. So $S$ is a (semi-local) regular ring, and this implies that $S=\overline{R}$ (recall $R \subseteq S \subseteq \overline{R}$).

It follows, since $A$ has minimal multiplicity, that $\n^2=\zeta \n$. Setting $\zeta=(\alpha, x)$ with $\alpha \in \m$ and $x\in \fkc$, and using the natural projection $p : A\to R, ~(\alpha, x) \mapsto \alpha$, we deduce that $\m^2=\alpha \m$, i.e., $R$ has minimal multiplicity. In particular, by Theorem \ref{intro1.2}, we see that $R$ is an Arf ring.

$(2) \Rightarrow (1)$ Assume $R$ has minimal multiplicity and $S = \overline{R}$. Notice, for each maximal ideal $\M$ of $S$, it follows that $\e(S_{\M}\times S_{\M})=2$ since $S_{\M}$ is regular.
Therefore it suffices to prove $A$ has minimal multiplicity; see Theorem \ref{intro1.2} and recall $A[L] = S \times S$.  

As $R$ has minimal multiplicity, we can pick an element $\alpha \in \m$ such that $\m^2 = \alpha \m$. Then, since $\m$ is stable, we have $c \subseteq R \subseteq R[\frac{\m}{\alpha}]=R^{\m}=m :_{\rmQ(R)} \m =B$. Thus $c$ is an ideal of $B$. As $R[\frac{\m}{\alpha}]=\frac{\m}{\alpha} \in B$, it follows that $\frac{\m}{\alpha} \cdot \fkc \subseteq \fkc$, i.e., $\m \cdot \fkc = \alpha \cdot \fkc$. Therefore we have $\n^2=\m^2 \times \m \fkc=\alpha \m \times \alpha \m= (\alpha, 0)\cdot (\m \times \fkc)$, i.e., $A$ has minimal multiplicity. Now, by Theorem \ref{intro1.2}, $A$ is Arf.
%Note that $R$ is analytically unramified, since $S$ is a module-finite extension over $R$. For every $M \in \Max S$, $S_M$ is a discrete valuation ring, whence 
%$\e(S_M \times S_M) = 2$. Let us take $\alpha \in \m$ such that $\m^2 = \alpha \m$. We then have
%$$
%\c \subseteq R \subseteq \frac{\m}{\alpha} = \m : \m =B \subseteq \overline{R} = S 
%$$
%which yields that $\c$ is an ideal of $B$. Since $\m \c = \alpha \c$, we have
%$$
%(\m \times \c)^2 = \m^2 \times \m \c = (\alpha, 0)\cdot (\m \times \c).
%$$
%Hence $A$ is an Arf ring.
\end{proof}

Here is a consequence of Theorem \ref{4.1} that gives a useful criterion for the idealization $R \ltimes \m$ to be almost Gorenstein Arf. Recall, if $R$ is an almost Gorenstein ring that is not Gorenstein, then $\fkc = R:_{\rmQ(R)}S=\m$; see \ref{2.3} and \cite[3.16]{GMP}.

\begin{cor}\label{3} $R \ltimes \m$ is an almost Gorenstein Arf ring if and only if $\m \overline{R} \subseteq R$.
%When this is the case, $\m:\m = \overline{R}$. 
\end{cor}

\begin{proof} Let $A=R \ltimes \m$ and $\n=\m \times \m$.

Assume first $A$ is an almost Gorenstein Arf ring. Then, by Theorem \ref{2.4},  the endomorphism algebra $\n :_{\rmQ(A)}\n$ is Gorenstein, where $\n=\m \times \m$. Let $\mathcal{N}$ be a maximal ideal of $B=\m:_{\rmQ(R)}\m$. Then, since $\n :_{\rmQ(A)}\n= B\times B$, and $\displaystyle{(\n :_{\rmQ(A)}\n) /(\mathcal{N} \times B)=B/\n}$, we see $\mathcal{N} \times B$ is a maximal ideal of $\n :_{\rmQ(A)}\n$. Moreover, the localization $\displaystyle{(\n :_{\rmQ(A)}\n)_{\mathcal{N} \times B}=B_{\n} \times B_{\n}}$ has multiplicity two; see Corollary \ref{3.4}. Therefore $B$ is regular, and $B=\overline{R}$.

Conversely, assume $\m \overline{R} \subseteq R$. If $R$ is regular, then $\m \cong R \cong \K_R$ so that $R \ltimes \m$ is Gorenstein \cite{Reiten}. So we may assume $R \neq \overline{R}$. Then $\m\overline{R} \neq R$. Since $\overline{R} \subseteq  \m:_{\rmQ(R)}\m=B$, we have $B=\overline{R}$. 
Then it follows that $\n :_{\rmQ(A)}\n = B \times B$. 
In particular, $\e((\n :_{\rmQ(A)}\n)_{\N})=2$ for each maximal ideal $\N$ of $B \times B$. So, by Corollary \ref{3.4}, we conclude $A$ is an almost Gorenstein Arf ring.
%We may assume $R$ is not a regular local ring. Then $\overline{R} = \m:\m$, so that $\n :\n = \overline{R}\ltimes \overline{R}$. Hence $\e((\n :\n)_N)=2$ for every $N \in \Max (\n:\n)$. Therefore, $R:\m$ is an almost Gorenstein Arf ring by Corollary \ref{3.4}.
%
%Hence we may assume $R$ is not Gorenstein. Then $R \ltimes \m=R \ltimes \fkc$.Assume first $R \ltimes \m$ is an almost Gorenstein Arf ring. If $R \ltimes \m$ is Gorenstein, then it follows \cite[6.6]{GMP} that $R$ is regular and hence $\m \overline{R}=\m R \subseteq R$. So assume Then it follows from Theorem \ref{4.1} that $R$ has minimal multiplicity and $S = \overline{R}$. Therefore \cite[3.11]{GMP} yields $\m \overline{R} \subseteq R$. \todo[inline]{HOW Theorem \ref{4.1} applies?} 
\end{proof}

Here is an application of Corollary \ref{3}.

\begin{ex} \label{son1} Let $(S, \n)$ be a regular local ring of dimension $d$ with $d\geq 3$. Set $R=S/I$ where $\displaystyle{I=\bigcap\limits^d_{i=1} (x_1, \ldots, \widehat{x_i}, \ldots, x_d})$, where $\widehat{x_i}$ is used to remove $x_i$ from the ideal. Then $R$ is an almost Gorenstein local ring with $B=\overline{R}$; see \cite[5.3]{GMP}. Since $B$ is Gorenstein, it follows from Theorem \ref{2.4} that $R$ has minimal multiplicity. Hence the maximal ideal $\m$ of $R$ is stable, so that $R_1=\overline{R}$. This implies all the blowups of $R$ equal to $\overline{R}$, and have multiplicity one. Now, by Proposition \ref{3.2}, $R$ is Arf. Since $\m \overline{R}=\m B \subseteq \m$, we conclude that $R \ltimes \m$ is an almost Gorenstein Arf ring.

%In particular, as a special case, we see that the ring $R \ltimes \m$ is an almost Gorenstein Arf ring, where $R = k[\![ x,y,z]\!]/[(x,y) \cap (y,z) \cap (z,x)]$. 
\end{ex}

In the next example we find out an idealization ring of the form $R \ltimes \m$ which is almost Gorenstein, but neither $R$ nor $R \ltimes \m$ is Arf.

\begin{ex} \label{son2} Let $R=k[[t^4, t^5, t^6]]$. Then 
$
R \ltimes \m \cong k[[X, Y, Z, U, V, W]]/I, %k[[t^4, t^5, t^6]] \ltimes (t^4, t^5, t^6)
$
where $I$ is the sum of the ideals $(YU-XV, ZU-XW, ZU-YV, ZV-YW, X^2U-ZW)$, 
$(X^3-Z^2, Y^2-ZX)$ and $(U, V, W)^2$.
Since $R$ is Gorenstein, we know from \cite[6.5]{GMP} that $R \ltimes \m$ is almost Gorenstein. Furthermore, since $t^7 \in \m \overline{R} \nsubseteq R$, we conclude from Corollary \ref{3} that $R \ltimes \m$ is not an Arf ring.
\end{ex}
%(X^3-Z^2, Y^2-ZX) + (U, V, W)^2 + (YU-XV, ZU-XW, ZU-YV, ZV-YW, X^2U-ZW).

%Let $k$ be an infinite field. We consider \vspace{-0.5em} 
%$$
%A=k[[X, Y, Z, U, V, W]]/I
%$$
%where \vspace{-0.5em}{\scriptsize 
%$$
%I=(X^3-Z^2, Y^2-ZX) + (U, V, W)^2 + (YU-XV, ZU-XW, ZU-YV, ZV-YW, X^2U-ZW).
%$$}

\section*{Acknowledgements}

Part of this work was completed when Celikbas visited the Meiji University in May and June 2017, and Taniguchi visited West Virginia University in February and March 2018. Taniguchi, partially supported by JSPS Grant-in-Aid for Young Scientists (B) 17K14176, is grateful for the kind hospitality of the WVU Department of Mathematics. Celikbas, supported by the Japan Society for the Promotion of Science (JSPS) Grant-in-Aid for Scientific Research (C) 26400054, are grateful for the kind hospitality of the Meiji Department of Mathematics.

Goto was partially supported by JSPS Grant-in-Aid for Scientific Research (C) 16K05112.

%%%%%%%%%%%%%%%%%%%%%%%%%%%%%%%%%%%%%%%%%%%%%%%%%%%%%%%%%%%%
%\addcontentsline{toc}{section}{references}

\end{document}